\RequirePackage{fix-cm}
\documentclass[smallextended]{svjour3} % onecolumn (second format)
\smartqed % flush right qed marks, e.g. at end of proof
\usepackage{graphicx}
\usepackage{etex,bm}
\usepackage{amsmath,amssymb,mathrsfs,mathtools}
\usepackage{amsfonts}
\usepackage{cite}
\usepackage[colorlinks=true,citecolor=blue,linkcolor=blue]{hyperref}
\usepackage[margin=1in]{geometry}
\usepackage{multirow,booktabs}
\usepackage{empheq}
\usepackage{array}
\usepackage{caption}
\usepackage{subcaption}
\usepackage{relsize}
\renewcommand{\figurename}{Figure}

\renewenvironment{proof}{\paragraph{Proof:}}{\hfill$\square$}
\usepackage[utf8]{inputenc}
\begin{document}

\title{Uncertain random geometric programming problems %\thanks{Grants or other notes
%about the article that should go on the front page should be
%placed here. General acknowledgments should be placed at the end of the article.}
}
%\subtitle{Do you have a subtitle?\\ If so, write it here}

\titlerunning{Uncertain random geometric programming problems} % if too long for running head

\author{Tapas Mondal \and
 Akshay Kumar Ojha \and Sabyasachi Pani%etc.
}

%\authorrunning{Mondal et al.} % if too long for running head

\author{Tapas Mondal* \and Akshay Kumar Ojha \and Sabyasachi Pani }
\institute{Tapas Mondal* (Corresponding author) \at
	School of Basic Sciences, Indian Institute of Technology Bhubaneswar \\
	Tel.: +918345939869\\
	%Fax: +123-45-678910\\
	\email{tm19@iitbbs.ac.in} % \\
	% \emph{Present address:} of F. Author % if needed
	\and
	Akshay Kumar Ojha \at
	School of Basic Sciences, Indian Institute of Technology Bhubaneswar \\
	\email{akojha@iitbbs.ac.in}
	\and
	Sabyasachi Pani \at
	School of Basic Sciences, Indian Institute of Technology Bhubaneswar \\
	\email{spani@iitbbs.ac.in}
}

\date{Received: date / Accepted: date}
% The correct dates will be entered by the editor

\maketitle

\begin{abstract}
In this paper, we introduce a deterministic formulation for the geometric programming problem, wherein the coefficients are represented as independent linear-normal uncertain random variables. To address the challenges posed by this combination of uncertainty and randomness, we introduce the concept of an uncertain random variable and present a novel framework known as the linear-normal uncertain random variable. Our main focus in this work is the development of three distinct transformation techniques: the optimistic value criteria, pessimistic value criteria, and expected value criteria. These approaches allow us to convert a linear-normal uncertain random variable into a more manageable random variable. This transition facilitates the transformation from an uncertain random geometric programming problem to a stochastic geometric programming problem. Furthermore, we provide insights into an equivalent deterministic representation of the transformed geometric programming problem, enhancing the clarity and practicality of the optimization process. To demonstrate the effectiveness of our proposed approach, we present a numerical example.\\
\keywords{Stochastic programming \and Uncertainty modelling \and Linear-normal uncertain random variable \and Geometric programming}
% \PACS{PACS code1 \and PACS code2 \and more}
\subclass{90C15 \and 90C30 \and 90C46 \and 90C47 \and 49K45}
\end{abstract}
\section{Introduction}\label{sec.1}
Geometric programming (GP) is a highly efficient approach for addressing nonlinear optimization problems when the objective and constraint functions are in posynomial form. In 1967, Duffin et al. \cite{Duffin 1967} first introduced the fundamental theories of the GP problem. In most cases, the GP method is employed for objective and constraint functions, which are in posynomial form. When all parameters in the GP problem are positive, except for the exponents, it is referred to as the posynomial problem. The GP approach transforms the original problem into a dual problem, making it more readily solvable.
\par In modern days, the GP problem is used in engineering design problems like circuit design \cite{Chu 2001,Hershenson 2001}, inventory modeling \cite{Mandal 2006,Roy 1997,Worrall 1982}, production planning \cite{Cheng 1991,Choi 1996,Islam 2007,Jung 2001,Kim 1998,Lee 1993}, risk management \cite{Scott 1995}, chemical processing \cite{Passy 1968,Ruckaert 1978,Wall 1986}, information theory \cite{Chiang 2005}, and structural design \cite{Gupta 1986}.
\par The classical GP problem is considered under the assumption that parameters are well-defined and fixed. Many researchers \cite{Avriel 1975,Beightler 1976,Duffin 1967,Duffin 1973,Fang 1988,Kortanek 1992,Kortanek 1997,Maranas 1997,Rajgopal 1990,Rajgopal 1992,Rajgopal 2002} have devised effective and functional algorithms designed for this traditional GP problem, where parameters are exact and specific. In reality, though, the values associated with GP parameters often lack certainty and precision. Consequently, various approaches have been crafted to address the challenge of GP in scenarios where parameter values are uncertain and less precise. Avriel and Wilde \cite{Avriel 1969}, for instance, introduced a new dimension with the stochastic GP problem, wherein they explored a GP scenario featuring objective and constraint function coefficients as nonnegative random variables. Dupačová \cite{Dupačová 2010} developed a metal cutting model in GP form and presented stochastic sensivity analysis. In 2016, Liu et al. \cite{Liu 2016} solved the stochastic GP problem with joint probabilistic constraints. Also, Liu et al. \cite{Liu 2020} came up with a new way to solve the joint rectangular chance-constrained GP problem with random parameters that are elliptically distributed and independent from each other. In 2021, Shiraz et al. \cite{Shiraz 2021} solved the stochastic GP problem with joint chance constraints based on copula theory. In this article, the coefficients of each term in the objective and constraint functions of the GP problem are taken as dependent random variables. Furthermore, the GP problem has been developed when the parameters are taken as interval types. Liu \cite{Liu 2006} developed a technique for determining the possible range of objective values in posynomial GP problems when coefficients and exponents have varying intervals. Additionally, Mahapatra and Mandal \cite{Mahapatra 2012} tackled the GP problem using interval-valued coefficients within a parametric framework.
\par Over the past few decades, the GP problem has been developed in fuzzy environments. In 1993, Cao \cite{Cao 1993} extended the GP problem to an imprecise environment, along with interval and fuzzy coefficients. Later on, the same author \cite{Cao 1997} proposed the GP problem with T-fuzzy coefficients. Mandal and Roy \cite{Mandal 2006} solved the GP problem with L-R fuzzy coefficients. Yang and Cao \cite{Yang 2007} made significant contributions to the fuzzy relational GP under monomial. Liu \cite{Liu 2007} developed the GP problem with fuzzy parameters. Shiraz et al. \cite{Shiraz 2017} used possibility, necessity, and credibility approaches to solve the fuzzy chance-constrained GP problem. Under the rough-set theory, Shiraz and Fukuyama \cite{Shiraz 2018} developed the GP problem.
\par Recently, Liu \cite{Liu 2015} pioneered uncertainty theory, a new and developing field of mathematics. To solve the GP problem, several researchers developed an uncertainty-based framework based on uncertainty theory. Shiraz et al. \cite{Shiraz 2016} first considered the GP problem in an uncertain environment. Based on uncertainty theory, the authors developed the deterministic form of an uncertain GP problem under normal, linear, and zigzag uncertainty distributions. In 2022, Mondal et al. \cite{Mondal 2022} developed a procedure for solving the GP problem with uncertainty. The authors derived the equivalent deterministic form of the GP problem under triangular and trapezoidal uncertainty distributions. In addition, the same authors \cite{Mondal 2023} proposed triangular and trapezoidal two-fold uncertainty distributions and developed three reduction methods to reduce two-fold uncertainty distributions into single-fold uncertainty distributions. Chassien and Goerigk \cite{Chassein 2019} introduced a robust GP problem with polyhedral and interval uncertainty. In 2023, Fontem \cite{Fontem 2023} demonstrated how to set bounds for a worst-case chance-constrained GP problem with random exponent parameters.
\par It is observed that the GP problem is taken into account under uncertainty distributions whenever the parameters of uncertainty distributions are taken in deterministic form. However, in reality, the parameters of an uncertainty distribution are not always deterministic. So the obvious question is: how can we deal with an uncertain variable when its parameters are random variables instead of constant values? If the parameters of an uncertain variable are random variables instead of constant values, then it is called an uncertain random variable. In 2013, Liu \cite{Liu 2013a} proposed an uncertain random variable that is a mixture of uncertainty and randomness. Further, the same author \cite{Liu 2013b} applied the concept of uncertain random variables to an optimization problem called uncertain random programming. In 2021, Liu et al. \cite{Liu 2021} developed portfolio optimization for uncertain random returns. In addition, Chen et al. \cite{Chen 2021} and Li et al. \cite{Li 2022} solved portfolio optimization with uncertain random parameters. Zhou et al. \cite{Zhou 2014} introduced multi-objective optimization in uncertain random environments. Wang et al. \cite{Wang 2015} developed an inventory model with uncertain random demand. Ke et al. \cite{Ke 2015} solved a decentralized decision-making production control problem with uncertain random parameters. In 2018, Qin \cite{Qin 2018} introduced uncertain goal programming. Gao et al. \cite{Gao 2018b} proposed the concept of an uncertain random variable in a complex form. Ahmadzade et al. \cite{Ahmadzade 2020} and Gao et al. \cite{Gao 2018a} introduced the convergence concept of distribution for an uncertain random variable.
\par In most of the optimization problems with uncertain random parameters, the authors considered some of the parameters to be uncertain variables and some to be random variables. From the literature survey, it is observed that there has been a lot of research on the crisp and fuzzy GP problem. In addition, an optimization problem is developed when some of parameters are uncertain variables and some of random variables. To the best of our knowledge, no previous work on the GP problem with uncertain random coefficients has been done. As a result, we attempt to investigate the GP problem with uncertain random coefficients. The following are the main contributions we made to this study: 
\begin{description}
	\item[I.] We introduce a novel deterministic formulation for the GP problem, where coefficients are uniquely represented as independent linear-normal uncertain random variables. 
	\item[II.] To tackle uncertainty and randomness, we introduce the concept of an uncertain random variable and propose a novel framework called the linear-normal uncertain random variable.
	\item[III.] We develop three distinct transformation methods, each with its own unique approach: the optimistic value criteria, pessimistic value criteria, and expected value criteria. These methods are instrumental in converting complex linear-normal uncertain random variables into more manageable, tractable random variables. 
	\item[IV.] The transformation methods enable us to convert a linear-normal uncertain random variable into a more tractable random variable, facilitating the transition from an uncertain random GP problem to a stochastic GP problem.
	\item[V.] We provide insights into an equivalent deterministic representation of the transformed GP problem, adding clarity and practicality to the optimization process.
	\item[VI.] To demonstrate the effectiveness and real-world applicability of our proposed approach, we present a comprehensive numerical example, offering tangible evidence of its efficacy in addressing GP problems.
\end{description}
 In essence, our research presents a paradigm shift in addressing GP problems, introducing innovative concepts and methods that enhance understanding, clarity, and practicality in optimization processes. 
\par The rest of the paper is organized as follows: A linear-normal uncertain random variable is proposed in Section \ref{sec.2}. Transformation methods are developed in Section \ref{sec.3}. The GP problem with linear-normal uncertain variable coefficients is considered, and the deterministic formulation is developed in Section \ref{sec.4}. In Section \ref{sec.5}, a numerical example is given. Finally, a conclusion on this work is incorporated in Section \ref{sec.6}.
\section{Linear-normal uncertain random variable}\label{sec.2} 
In this section, we introduce the concept of a linear-normal uncertain random variable. To begin, let us delve into some foundational ideas from both probability theory and uncertainty theory.
\subsection{Probability theory}
Probability theory is a field of mathematics focused on exploring the characteristics of unpredictable events. Within probability theory, the concept of a random variable plays an important role. To grasp the essence of a random variable, it is essential to familiarize ourselves with the notions of probability measure and probability space.
\begin{definition}\cite{Liu 2004}\label{def.1}
	Consider the $\sigma-$algebra $L_p$ defined over a nonempty set $\Omega$. To establish the concept of a random variable, we must initially introduce the probability measure and probability space. The probability measure, denoted as $Pr:L_p\rightarrow[0,1],$ adheres to the following axioms:
	\begin{description}
		\item[I.] (Normality Axiom) $Pr\{\Omega\}=1$.
		\item[II.] (Nonnegativity Axiom) $Pr\{A\}\ge 0$ for any event $A\in L_p$.
		\item[III.] (Additivity Axiom) $Pr\big\{ \bigcup\limits_{i=1}^{\infty}A_i \big\} = \sum\limits_{i=1}^{\infty}Pr\{A_i\}$ for every countable sequence of events $\{A_i\}_{i=1}^{\infty}$.
	\end{description}
\end{definition}
	\par Furthermore, the trio denoted as $(\Omega,L_p,Pr)$ constitutes the probability space. The notion of a random variable proves highly valuable in the context of probability description. Here, we provide the definition of a random variable.
	\begin{definition}\cite{Liu 2004}\label{def.2}
		A measurable function, denoted as $\omega:(\Omega,L_p,Pr)\to \mathbb{R}$, is referred to as a random variable. In this context, $(\Omega,L_p,Pr)$ represents the probability space, and $\mathbb{R}$ is the set of real numbers. In simpler terms, for any Borel set $B$, the set ${\omega\in B}$ is considered an event.
	\end{definition}
	\par On the other hand, one can distinctly define a random variable through its associated probability distribution, which holds essential information about the random variable. In many instances, having knowledge of the probability distribution is more than enough, obviating the need to know the random variable directly. Below is the mathematical definition of probability distribution.
	\begin{definition}\cite{Liu 2004}\label{def.3}
		If $\Phi_p$ stands as the probability distribution associated to a random variable denoted as $\omega$, then $\Phi_p$ is formally defined as
		\begin{equation*}
			\Phi_p(x)=Pr\{\omega\le x\},\forall x\in \mathbb{R}.
		\end{equation*}
	\end{definition}
	\par Additionally, the following theorem supports the requirement that is both necessary and sufficient for a probability distribution function.
	\begin{theorem}\cite{Liu 2004}\label{thm.1}
		A function $\Phi_p:\mathbb{R}\to[0,1]$ qualifies as a probability distribution if and only if it satisfies the conditions of being monotonically increasing, right-continuous, and having the limits $\underset{x\to -\infty}\lim\Phi_p(x)=0$ and $\underset{x\to +\infty}\lim\Phi_p(x)=1$.
	\end{theorem}
\par  Normal random variables play an important role in order to define the linear-normal uncertain random variable. Here is the definition of the normal random variable.
\begin{definition}\cite{Liu 2004}\label{def.4}
	Let $\Phi_p$ be the probability distribution corresponding to a random variable $\omega$. The random variable $\omega$ is said to be normal if and only if $\Phi_p$ is defined as\\
	\[\Phi_p(x)=\frac{1}{2}\bigg[1+\text{erf}\bigg(\frac{x-\mu}{\sigma\sqrt{2}}\bigg)\bigg], x\in\mathbb{R},
	\]
	where $\mu,\sigma$ are real numbers with $\sigma>0$, and  $\text{erf}(z)=\frac{2}{\sqrt{\pi}}\int\limits_{0}^{z}e^{-t^2}dt.$
	If $\omega$ is a normal random variable with the parameters $\mu$ and $\sigma$, then we write it as $\omega\sim \mathcal{N}(\mu,\sigma).$ 
\end{definition}
\begin{remark}
If $\mu=0$ and $\sigma=1$, then the random variable $\omega\sim \mathcal{N}(0,1)$ is known as a standard normal random variable. Its distribution function is 
\[\Phi(x)=\frac{1}{2}\bigg[1+\text{erf}\bigg(\frac{x}{\sqrt{2}}\bigg)\bigg], x\in\mathbb{R}.
\]
\end{remark}
\subsection{Uncertainty theory} 
 Uncertainty theory is an emerging mathematical discipline that employs uncertain variables to depict the attributes of uncertainties and represent degrees of belief. In order to introduce the concept of uncertain variables, it is essential to grasp the notions of uncertain measure and uncertainty space.
\begin{definition}\cite{Liu 2015}\label{def.5}
	Consider the $\sigma-$algebra $L_u$ defined over a nonempty set $\Gamma$. To introduce the concept of an uncertain variable, we must initially define the uncertain measure and the concept of an uncertainty space. The uncertain measure is represented as the set function $\mathcal{M}:L_u\rightarrow[0,1],$ which satisfies to the following axioms:
	\begin{description}
		\item[I.] (Normality Axiom) $\mathcal{M}\{\Gamma\}=1$.
		\item[II.] (Duality Axiom) $\mathcal{M}\{A\}+\mathcal{M}\{A^c\}=1$ for any event $A\in L_u$.
		\item[III.] (Subadditivity Axiom) $\mathcal{M}\big\{ \bigcup\limits_{i=1}^{\infty}A_i \big\} \leq \sum\limits_{i=1}^{\infty}\mathcal{M}\{A_i\}$ for every countable sequence of events $\{A_i\}_{i=1}^{\infty}$.
	\end{description}
\end{definition}
\par Furthermore, the trio denoted as $(\Gamma,L_u,\mathcal{M})$ constitutes the uncertainty space. The concept of the uncertain variable is highly valuable in the context of describing uncertainties. Below is the formal definition of the uncertain variable.
\begin{definition}\cite{Liu 2015}\label{def.6}
	A measurable function denoted as $\xi:(\Gamma,L_u,\mathcal{M})\to \mathbb{R}$ is referred to as an uncertain variable. In this context, $(\Gamma,L_u,\mathcal{M})$ represents the uncertainty space, and $\mathbb{R}$ corresponds to the set of real numbers. Put simply, for any Borel set $B$, the set ${\xi\in B}$ is considered an event.
\end{definition}
\par Alternatively, an uncertain variable is identified through its associated uncertainty distribution, which carries uncertain information regarding the uncertain variable. In many instances, having knowledge of the uncertainty distribution is more than sufficient, obviating the need to know the uncertain variable directly. Here is the mathematical definition of uncertainty distribution.
\begin{definition}\cite{Liu 2015}\label{def.7}
	If $\Phi_u$ represents the uncertainty distribution corresponding to an uncertain variable $\xi$, then $\Phi_u$ is defined as
	\begin{equation*}
		\Phi_u(x)=\mathcal{M}\{\xi\le x\},\forall x\in \mathbb{R}.
	\end{equation*}
\end{definition}
\par In addition, the requisite and complete condition for an uncertainty distribution function is provided in the following theorem.
\begin{theorem}\cite{Liu 2015}\label{thm.2}
	A function $\Phi_u:\mathbb{R}\to[0,1]$ qualifies as an uncertainty distribution if and only if it displays monotonic increasing behavior, excluding the cases where $\Phi_u(x)=0$ or $\Phi_u(x)=1.$
\end{theorem}
\par Describing uncertain information geometrically in a two-dimensional space using the uncertainty distribution is straightforward. Real-world decision-making systems often employ various uncertain variables, such as normal, linear, zigzag, log-normal, triangular, and trapezoidal. Linear uncertain variables are particularly significant in defining linear-normal uncertain random variables. Here is the definition of a linear uncertain variable.
\begin{definition}\cite{Liu 2015}\label{def.8}
	Let $\Phi_u$ represents the uncertainty distribution corresponding to an uncertain variable $\xi$. The uncertain variable $\xi$ is classified as linear if and only if $\Phi_u$ is of the following form:\\
	\[\Phi_u(x)=\left\{
	\begin{array}{ll}
		0, & x\leq a;\\
		\frac{x-a}{b-a}, & a\leq x\leq b;\\
		1, & x \geq b;\\
	\end{array}
	\right.
	\]
where $a$ and $b$ are real numbers with $a < b$. When $\xi$ is a linear uncertain variable with parameters $a$ and $b$, we denote it as $\xi\sim \mathcal{L}(a,b)$.
\end{definition}
\subsection{Uncertain random variable}
An uncertain random variable is an uncertain variable whose parameters are random variables instead of constant values. For better clarification, in a linear uncertain variable $\xi\sim \mathcal{L}(a,b)$, the parameters $a$ and $b$ are taken as constants. In this scenario, the obvious question is: how can we deal with the linear uncertain variable $\xi\sim \mathcal{L}(a,b)$ if its parameters $a$ and $b$ are random variables instead of constant values? In order to describe this kind of variable, we need to define an uncertain random variable.
\begin{definition}\label{def.9}
	An uncertain random variable $\tilde{\xi}$ is a function from the uncertainty space $(\Gamma,L_u,\mathcal{M})$ to the set of random variables such that $Pr\{\tilde{\xi}(\omega)\in B\}$ is a measurable function of $\omega$ for any Borel set $B$ of $\mathbb{R}.$
\end{definition}
\begin{remark}
	In broad terms, an uncertain random variable can be described as a measurable function that operates within the uncertainty space and maps to the set of random variables. In simpler terms, it is an element characterized by uncertainty, which can take on values that belong to the category of random variables.
	\end{remark}
\par In this research, we propose a linear-normal uncertain random variable. The proposed definition is given below.
\begin{definition}\label{def.10}
	A linear-normal uncertain random variable is a linear uncertain variable whose parameters are normal random variables. More precisely, the uncertain random variable $\tilde{\xi}\sim \mathcal{L}(A,B)$ is said to be a linear-normal uncertain random variable if and only if the parameters $A$ and $B$ are normal random variables. 
\end{definition}
\begin{remark}
In Definition \ref{def.10}, the parameters $A$ and $B$ of the linear-normal uncertain random variable $\tilde{\xi}\sim \mathcal{L}(A,B)$ are taken as independent normal random variables.
\end{remark}	
\section{ Transformation of uncertain random variable to random variable}\label{sec.3}
In this section, we demonstrate the process of converting an uncertain random variable into a conventional random variable. To accomplish this, we employ three critical value criteria: optimistic, pessimistic, and expected value criteria. We begin by introducing the idea of critical values for uncertain variables, which Liu \cite{Liu 2015} first introduced in 2015. Here is the definition of these critical values.
\begin{definition}\cite{Liu 2015,Yang 2015}\label{def.11}
	If $\xi$ is an uncertain variable, then the optimistic, pessimistic, and expected values of $\xi$ are defined as follows:
	\begin{description}
		\item[ I.] $\xi_{\sup}(\alpha)=\sup\{r|\mathcal{M}\{\xi\geq r\}\geq \alpha\},\alpha\in (0,1);$
		\item[ II.] $\xi_{\inf}(\alpha)=\inf\{r|\mathcal{M}\{\xi\leq r\}\geq \alpha\},\alpha\in (0,1);$
		\item[ III.] $\xi_{exp}=\int\limits_{0}^{\infty}\mathcal{M}\{\xi\geq r\}dr-
		\int\limits_{-\infty}^{0}\mathcal{M}\{\xi\leq r\}dr.$
	\end{description}
\end{definition} 
\par Alternatively, Definition \ref{def.11} can be expressed using uncertainty distribution. The subsequent theorem, commonly referred to as the measure inversion theorem, proves to be highly valuable in this context.
\begin{theorem}\cite{Liu 2015,Yang 2015}\label{thm.3}
	Consider an uncertain variable $\xi$ with an associated uncertainty distribution $\Phi_u$. For any real number $r$, the following relationships hold: $\mathcal{M}\{\xi\leq r\}=\Phi_u(r)\text{ and } \mathcal{M}\{\xi\geq r\}=1-\Phi_u(r).$
\end{theorem}
\par Based on Theorem \ref{thm.3}, Definition \ref{def.11} can be redefined as
\begin{description}
	\item[ I.] $\xi_{\sup}(\alpha)=\sup\{r|1-\Phi_u(r)\geq \alpha\},\alpha\in (0,1);$
	\item[ II.] $\xi_{\inf}(\alpha)=\inf\{r|\Phi_u(r)\geq \alpha\},\alpha\in (0,1);$
	\item[ III.] $\xi_{exp}=\int\limits_{0}^{\infty}(1-\Phi_u(r))dr-
	\int\limits_{-\infty}^{0}\Phi_u(r)dr.$
\end{description}
\begin{remark}
	Integrating by parts, we have 
	$\int\limits_{0}^{\infty}(1-\Phi_u(r))dr-
	\int\limits_{-\infty}^{0}\Phi_u(r)dr=\int\limits_{-\infty}^{\infty}r{\Phi'_u}(r)dr.$ Therefore, the expected value of an uncertain variable $\xi$ can be determined by the formula given as $\xi_{exp}=\int\limits_{-\infty}^{\infty}r{\Phi'_u}(r)dr.$ 
\end{remark}
\par Based on the above operations, the following useful theorem is developed.
\begin{theorem}\cite{Liu 2015,Yang 2015}\label{thm.4}
	Let $\xi\sim \mathcal{L}(a,b)$ be a linear uncertain variable. The optimistic, pessimistic, and expected values of $\xi$ are as follows:
	\begin{description}
		\item[ I.] $\xi_{\sup}(\alpha)=\alpha a+(1-\alpha)b,\alpha\in (0,1);$
		\item[ II.] $\xi_{\inf}(\alpha)=(1-\alpha)a+\alpha b,\alpha\in (0,1);$
		\item[ III.] $\xi_{exp}=\frac{a+b}{2}.$
	\end{description} 
\end{theorem} 
\par It is observed that if the parameters of an uncertain variable are random variables instead of constant values, then the critical values of the uncertain variable are random variables instead of constant values. Based on this concept, we transform the uncertain random variable into a random variable using the critical values of uncertainty.  In particular, if $\tilde{\xi}\sim \mathcal{L}(A,B)$ is a linear-normal uncertain random variable, then based on Theorem \ref{thm.4}, the transformed random variables using optimistic, pessimistic, and expected value criteria are as follows:
\begin{description}
	\item[ I.] $\tilde\xi_{\sup}(\alpha)=\alpha A+(1-\alpha)B,\alpha\in (0,1);$
	\item[ II.] $\tilde\xi_{\inf}(\alpha)=(1-\alpha)A+\alpha B,\alpha\in (0,1);$
	\item[ III.] $\tilde\xi_{exp}=\frac{A+B}{2}.$
\end{description} 
\par The main aim is to characterize the random variables $\tilde\xi_{\sup}(\alpha),$ $\tilde\xi_{\inf}(\alpha),$ and $\tilde\xi_{exp}.$ The characteristic function of a random variable is an important concept and plays a powerful role in this regard. Here is the definition of the characteristic function of a random variable.
\begin{definition}\cite{Liu 2004}\label{def.12}
	If $\omega$ is a random variable with probability distribution $\Phi_p$, then its characteristic function is defined as $\chi_{\omega}(t)=E[e^{it\omega }]=\int\limits_{-\infty}^{\infty}e^{itr}{\Phi'_p}(r)dr,t\in \mathbb{R},i=\sqrt{-1}.$
\end{definition} 
\par Using the characteristic function, we characterize the transformed random variables $\tilde\xi_{\sup}(\alpha),$ $\tilde\xi_{\inf}(\alpha),$ and $\tilde\xi_{exp}.$ The following theorem is helpful in this regard.
\begin{theorem}\label{thm.5}
If $\tilde{\xi}\sim \mathcal{L}(A,B)$ is a linear-normal uncertain random variable where $A\sim \mathcal{N}(\mu_A,\sigma_A)$ and $B\sim \mathcal{N}(\mu_B,\sigma_B)$ are two independent normal random variables, then the following results hold:
	\begin{description}
		\item[ I.] $\tilde\xi_{\sup}(\alpha)\sim \mathcal{N}\big(\alpha\mu_A+(1-\alpha)\mu_B,\sqrt{\alpha^2\sigma_A^2+(1-\alpha)^2\sigma_B^2}\big),\alpha\in (0,1);$
		\item[ II.] $\tilde\xi_{\inf}(\alpha)\sim \mathcal{N}\big((1-\alpha)\mu_A+\alpha\mu_B,\sqrt{(1-\alpha)^2\sigma_A^2+\alpha^2\sigma_B^2}\big),\alpha\in (0,1);$
		\item[ III.] $\tilde\xi_{exp}\sim \mathcal{N}\big(\frac{\mu_A+\mu_B}{2},\frac{\sqrt{\sigma_A^2+\sigma_B^2}}{2}\big).$
	\end{description} 
\end{theorem}
\begin{proof}
We prove only (I). The proofs of (II) and (III) are similar to the proof of (I).
\par We have $\tilde\xi_{\sup}(\alpha)=\alpha A+(1-\alpha)B,$ where $A\sim \mathcal{N}(\mu_A,\sigma_A)$ and $B\sim \mathcal{N}(\mu_B,\sigma_B)$ are two independent normal random variables. So, the characteristic functions of $A$ and $B$ are $\chi_A(t)=E[e^{itA}]$ and $\chi_B(t)=E[e^{itB}].$ Using Definition \ref{def.12}, we get $\chi_A(t)=e^{i\mu_At-\frac{1}{2}\sigma_A^2t^2}$ and $\chi_B(t)=e^{i\mu_Bt-\frac{1}{2}\sigma_B^2t^2}.$ Now, the characteristic function of the random variable $\tilde\xi_{\sup}(\alpha)$ is calculated as 
	\begin{equation*} 
	\begin{aligned}
		\chi_{\tilde\xi_{\sup}}(t;\alpha)&=E\big[e^{it\tilde\xi_{\sup}(\alpha)}\big]\\
		&=E\bigg[e^{it\big(\alpha A+(1-\alpha)B\big)}\bigg]\\
		&=E\big[e^{it\alpha A}\cdot e^{it(1-\alpha)B}\big]\\
		&=E\big[e^{it\alpha A}\big]E\big[ e^{it(1-\alpha)B}\big]\\
		&=\chi_A(\alpha t)\cdot\chi_B((1-\alpha)t)\\
		&=e^{i\mu_A(\alpha t)-\frac{1}{2}\sigma_A^2(\alpha t)^2}\cdot e^{i\mu_B((1-\alpha)t)-\frac{1}{2}\sigma_B^2((1-\alpha)t)^2}\\
		&=e^{i(\alpha \mu_A+(1-\alpha)\mu_B)t-\frac{1}{2}(\alpha^2\sigma_A^2+(1-\alpha)^2\sigma_B^2)t^2}.
	\end{aligned}
\end{equation*}
Therefore, the characteristic function of $\tilde\xi_{\sup}(\alpha)$ is $\chi_{\tilde\xi_{\sup}}(t;\alpha)=e^{i(\alpha \mu_A+(1-\alpha)\mu_B)t-\frac{1}{2}(\alpha^2\sigma_A^2+(1-\alpha)^2\sigma_B^2)t^2}.$ This shows that $\tilde\xi_{\sup}(\alpha)\sim \mathcal{N}\big(\alpha\mu_A+(1-\alpha)\mu_B,\sqrt{\alpha^2\sigma_A^2+(1-\alpha)^2\sigma_B^2}\big),$ which completes the proof.
\end{proof}
\par It is observed that the parameters of the transformed random variables $\tilde\xi_{\sup}(\alpha),$ $\tilde\xi_{\inf}(\alpha),$ and $\tilde\xi_{exp}$ are known. So, we can visualize their corresponding probability distribution functions. The probability distributions of $\tilde\xi_{\sup}(\alpha),$ $\tilde\xi_{\inf}(\alpha),$ and $\tilde\xi_{exp}$ are as follows:
	\begin{description}
	\item[ I.] \[\Phi_{\tilde\xi_{\sup}}(x;\alpha)=\frac{1}{2}\bigg[1+\text{erf}\bigg(\frac{x-\alpha\mu_A-(1-\alpha)\mu_B}{\sqrt{2\alpha^2\sigma_A^2+2(1-\alpha)^2\sigma_B^2}}\bigg)\bigg], x\in\mathbb{R},\alpha\in (0,1);
	\]
	\item[ II.] \[\Phi_{\tilde\xi_{\inf}}(x;\alpha)=\frac{1}{2}\bigg[1+\text{erf}\bigg(\frac{x-(1-\alpha)\mu_A-\alpha\mu_B}{\sqrt{2(1-\alpha)^2\sigma_A^2+2\alpha^2\sigma_B^2}}\bigg)\bigg], x\in\mathbb{R},\alpha\in (0,1);
	\]
	\item[ III.] \[\Phi_{\tilde\xi_{exp}}(x)=\frac{1}{2}\bigg[1+\text{erf}\bigg(\frac{2x-\mu_A-\mu_B}{\sqrt{2\sigma_A^2+2\sigma_B^2}}\bigg)\bigg], x\in\mathbb{R}.
	\]
\end{description} 
\par To understand the efficiency of the transformation methods, we provide an example of a linear-normal uncertain random variable below.
\begin{example}
	If $\tilde{\xi}\sim\mathcal{LN}(A,B)$ is a linear-normal uncertain random variable where $A\sim \mathcal{N}(3,1)$ and $B\sim \mathcal{N}(2,1)$ are two independent normal random variables, then
	\begin{description}
		\item[I.] the transformed random variable via $\alpha$ optimistic value criteria is $\tilde\xi_{\sup}(\alpha)\sim \mathcal{N}(\alpha+2,\sqrt{\alpha^2+(1-\alpha)^2})$ and its corresponding probability distribution is 
		\[\Phi_{\tilde\xi_{\sup}}(x;\alpha)=\frac{1}{2}\bigg[1+\text{erf}\bigg(\frac{x-\alpha-2}{\sqrt{2\alpha^2+2(1-\alpha)^2}}\bigg)\bigg], x\in\mathbb{R},\alpha\in (0,1);
		\]
		\item[II.] the transformed random variable via $\alpha$ pessimistic value criteria is $\tilde\xi_{\inf}(\alpha)\sim \mathcal{N}(3-\alpha,\sqrt{\alpha^2+(1-\alpha)^2})$ and its corresponding probability distribution is 
		\[\Phi_{\tilde\xi_{\inf}}(x;\alpha)=\frac{1}{2}\bigg[1+\text{erf}\bigg(\frac{x+\alpha-3}{\sqrt{2\alpha^2+2(1-\alpha)^2}}\bigg)\bigg], x\in\mathbb{R},\alpha\in (0,1);
		\]
		\item[III.] the transformed random variable via expected value criteria is $\tilde\xi_{exp}\sim \mathcal{N}(\frac{5}{2},\frac{\sqrt{2}}{2})$ and its corresponding probability distribution is 
		\[\Phi_{\tilde\xi_{exp}}(x)=\frac{1}{2}\bigg[1+\text{erf}\bigg(\frac{2x-5}{2}\bigg)\bigg], x\in\mathbb{R}.
		\]
	\end{description}
\end{example}
 Utilizing the optimistic, pessimistic, and expected value criteria, we depict the transformed probability distribution of the linear-normal uncertain random variable $\tilde{\xi}\sim\mathcal{LN}(A,B)$ in Figure \ref{fig:1}, Figure \ref{fig:2}, and Figure \ref{fig:3}, correspondingly, where $A\sim \mathcal{N}(3,1)$ and $B\sim \mathcal{N}(2,1)$. Figure \ref{fig:1} and Figure \ref{fig:2} are presented as solid illustrations because they essentially represent bivariate functions involving both $x$ and $\alpha$. To offer a clear visualization for different values of $\alpha\in (0,1)$, we present the transformed probability distributions using the optimistic and pessimistic value criteria in Figure \ref{fig:4} and Figure \ref{fig:5} as two-dimensional plots.
 \begin{remark}
 	The optimistic value criteria transformation method for a fixed $\alpha\in(0,1)$ is equivalent to the pessimistic value criteria transformation method for $(1-\alpha)$. In particular, at $\alpha = 0.5,$ the optimistic value criteria and pessimistic value criteria transformation methods are equivalent to the expected value criteria transformation method.
 \end{remark}
 \begin{figure}[htb] 
 	\centering 
 	\includegraphics[scale=0.2]{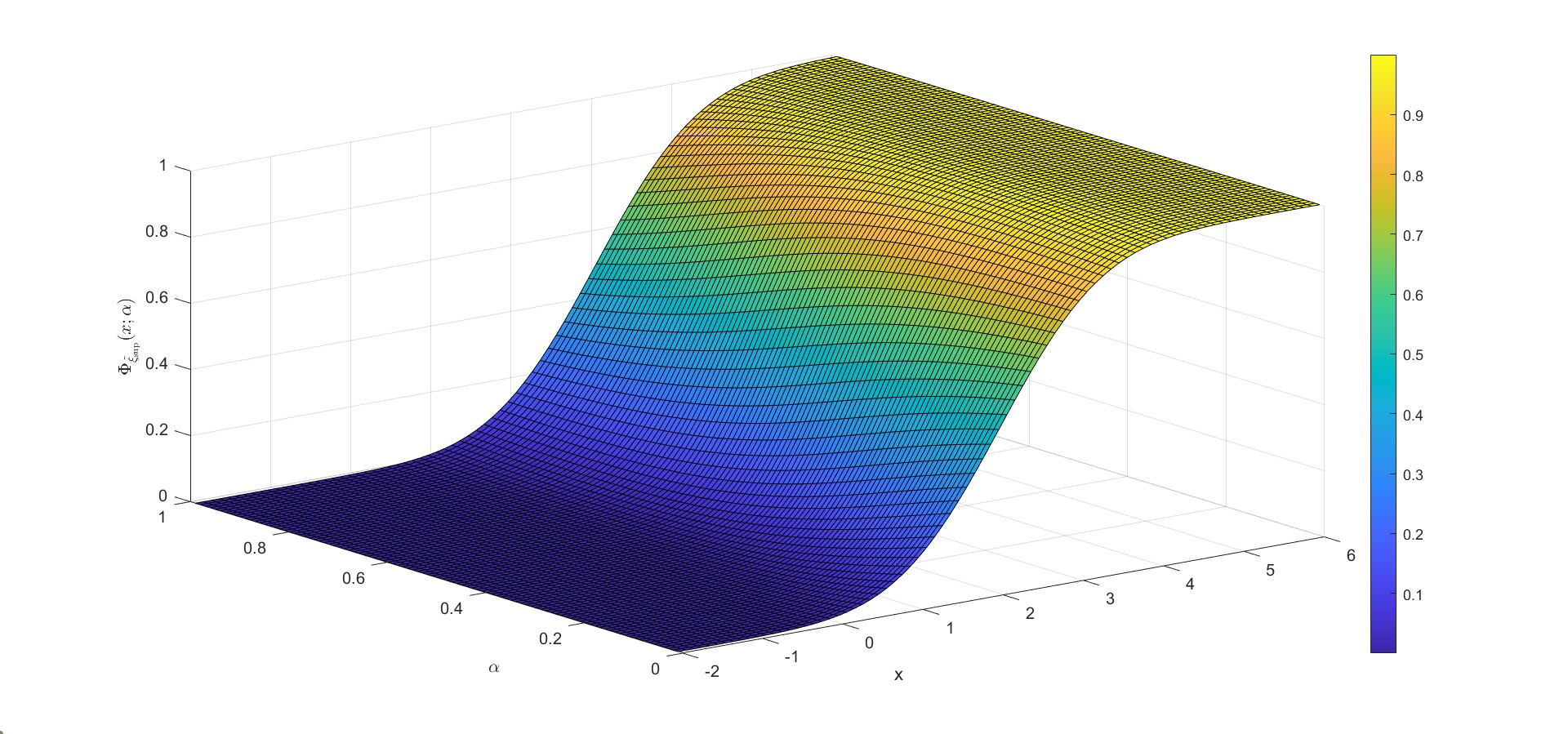} 
 	\caption{Transformed probability distribution of a linear-normal uncertain random variable via $\alpha$ optimistic value criteria.}
 	\label{fig:1}
 \end{figure}
 \begin{figure}[htb] 
 	\centering 
 	\includegraphics[scale=0.2]{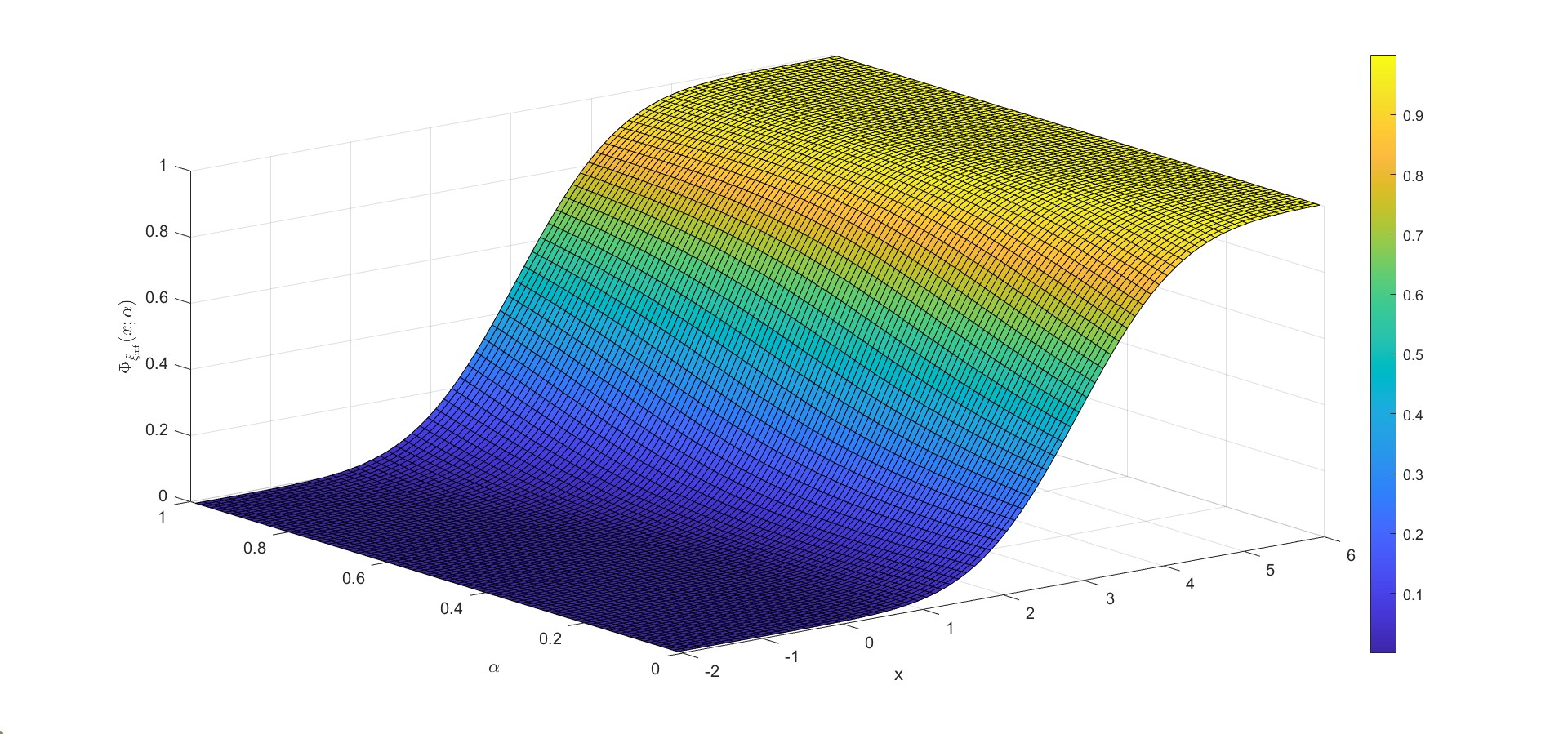} 
 	\caption{Transformed probability distribution of a linear-normal uncertain random variable via $\alpha$ pessimistic value criteria.}
 	\label{fig:2}
 \end{figure}
 \begin{figure}[htb] 
 	\centering 
 	\includegraphics[scale=0.2]{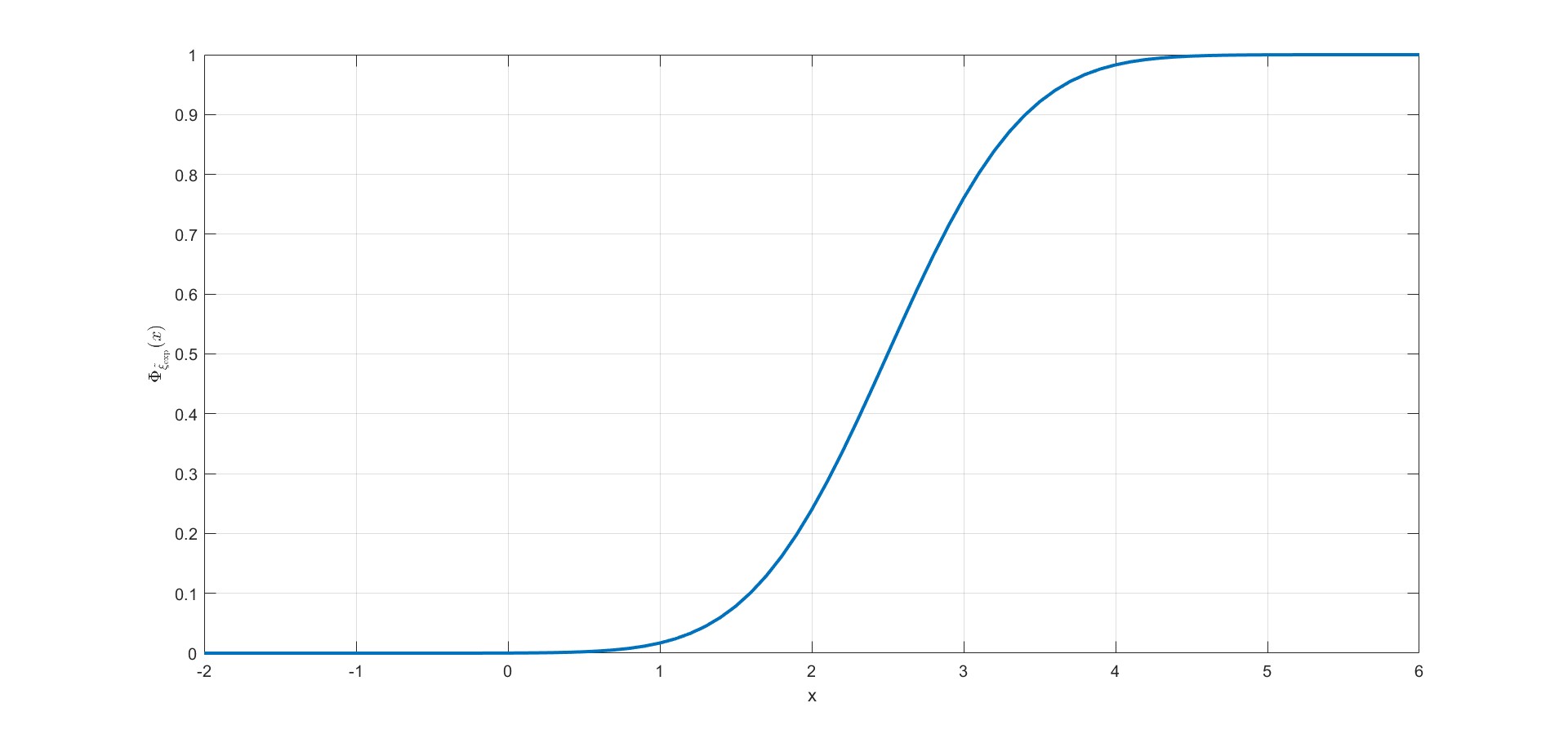} 
 	\caption{Transformed probability distribution of a linear-normal uncertain random variable via expected value criteria.}
 	\label{fig:3}
 \end{figure}
 \begin{figure}[htb]
 	\centering
 	\includegraphics[scale=0.2]{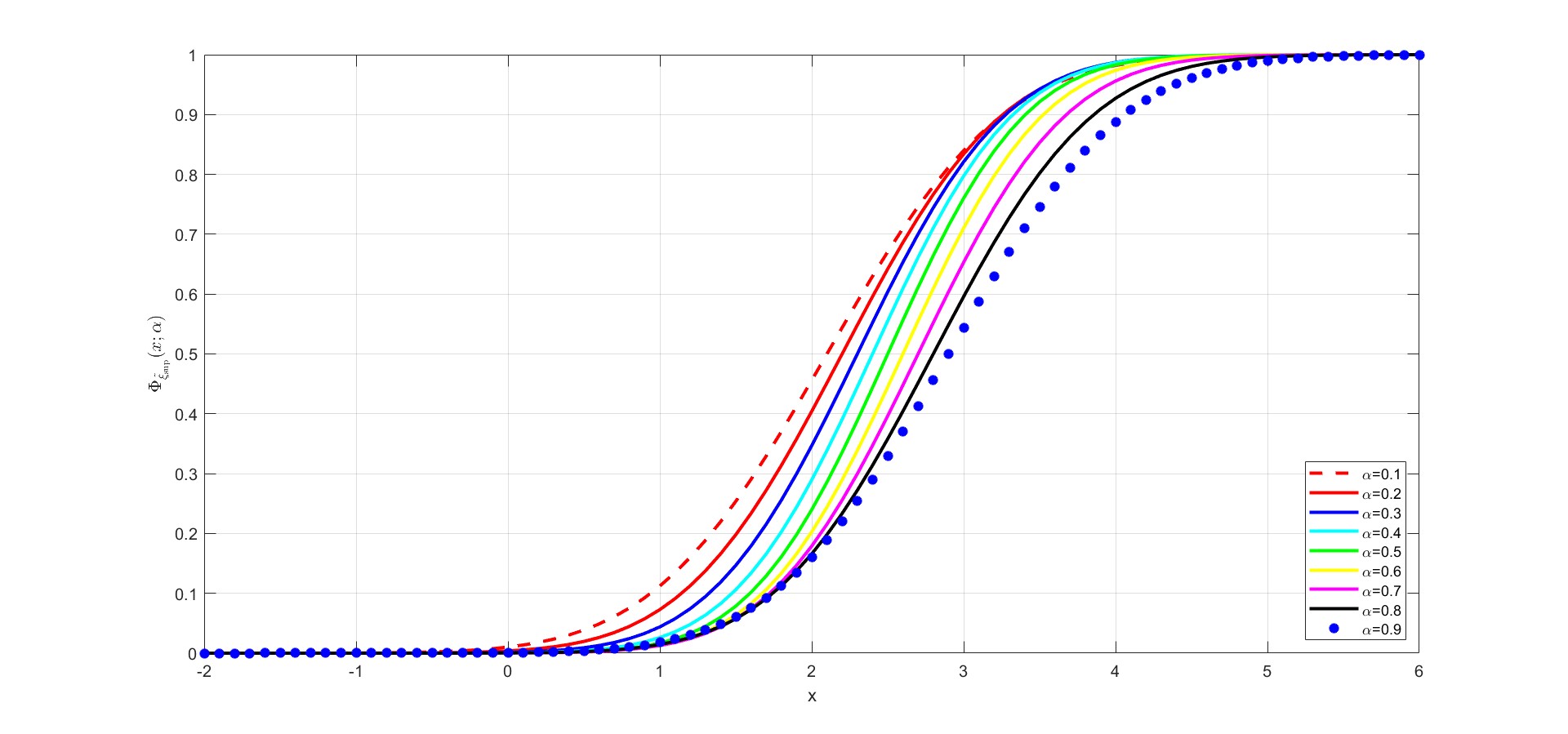}
 	\caption{Transformed probability distributions of a linear-normal uncertain random variable via optimistic value criteria for different values of $\alpha$.}
 	\label{fig:4}
 \end{figure}
 \begin{figure}[htb]
 	\centering
 	\includegraphics[scale=0.2]{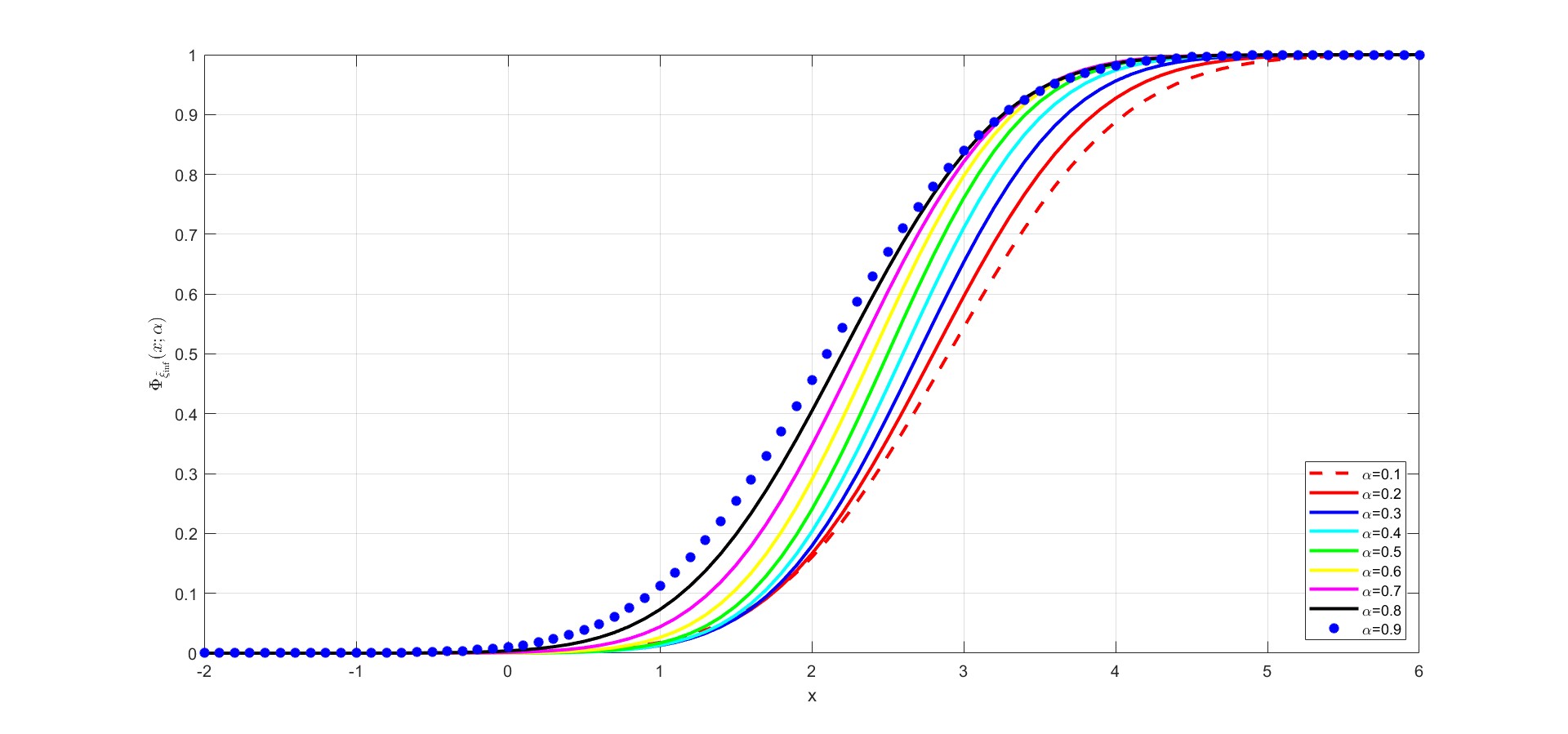}
 	\caption{Transformed probability distributions of a linear-normal uncertain random variable via pessimistic value criteria for different values of $\alpha$.}
 	\label{fig:5}
 \end{figure} 
\section{ Geometric programming problem with linear-normal uncertain random variable coefficients }\label{sec.4} 
	A conventional GP problem involves addressing a minimization problem where both the objective function and the constraints exhibit posynomial characteristics. Problems within the GP framework, in which all coefficients and decision variables are positive except for the exponents, are categorized as posynomial problems. A deterministic GP problem is outlined as
\begin{equation}\label{eq.1}
	\begin{aligned}
		&\min \quad f_0(\textbf{x})=\sum\limits_{i=1}^{N_0}\beta_{i0}\prod\limits_{j=1}^{n}x_j^{\alpha_{0ij}}\\
		&\text{s.t.}\\
		& \qquad f_k(\textbf{x})=\sum\limits_{i=1}^{N_k}\beta_{ik}\prod\limits_{j=1}^{n}x_j^{\alpha_{kij}}\leq 1,k=1,2,\ldots K,
	\end{aligned}
\end{equation}
where $\beta_{ik}>0, x_j>0, \alpha_{kij} \in \mathbb{R},\forall i,j,k$.
In this context, we have $N_0$ representing the total number of terms in the objective function, and $N_k$ representing the total number of terms in the $k^{th}$ constraint of Problem \ref{eq.1}. To formulate the dual for Problem \ref{eq.1}, let us introduce the variable $N$ as the total number of terms present in Problem \ref{eq.1}, given by $N=\sum\limits_{k=0}^{K}N_k$. Additionally, we use $\delta_{ik}$ as the dual variables, and $\lambda_k$ as the sum of dual variables present in the $k^{th}$ constraint, defined as $\lambda_k=\sum\limits_{i=1}^{N_k}{\delta_{ik}}$, where $k=0,1,2,\ldots,K$. Consequently, the dual problem for Problem \ref{eq.1} can be expressed as
\begin{equation}\label{eq.2}
	\begin{aligned}
		&	\max \quad V(\delta)=\prod_{i=1}^{N}\bigg(\frac{\beta_{ik}}{{\delta_{ik}}}\bigg)^{{\delta_{ik}}}\bigg(\lambda_k\bigg)^{\lambda_k}\\
		&	\text{s.t.}\\
		& \qquad \sum\limits_{i=1}^{N_0}{\delta_{i0}}=\lambda_0=1, \quad (\text{Normality condition})\\
		& \qquad\sum\limits_{i=1}^{N}{\delta_{ik}}{\alpha_{kij}}=0,j=1,2,\ldots,n. \quad (\text{Orthogonality conditions})
	\end{aligned}
\end{equation}
\par  In a GP problem, the primal problem is converted into its dual problem, which is easier to solve than the primal problem. The relationship between primal and dual problems is useful to find the primal decision variables. The relationship between primal and dual GP problems due to strong duality theorem \cite{Duffin 1967,Duffin 1973} is as follows:
\begin{equation}\label{eq.3}
	\begin{aligned}
		&\sum\limits_{j=1}^n{\alpha_{0ij} } \ln(x_j)=\ln \bigg(\frac{{\delta_{i0}}f_0(\textbf{x}) }{\beta_{i0}}\bigg), i=1,2,\ldots,N_{0},\\
		&\sum\limits_{j=1}^n{\alpha_{kij} } \ln(x_j)=\ln \bigg(\frac{\delta_{ik}} { \lambda_k \beta_{ik}}\bigg),i=1,2,\ldots,N_{k},k=1,2,\ldots,K.
	\end{aligned}
\end{equation}
\begin{remark}
	In a GP problem, the total number of terms minus the total number of decision variables minus one is known as the degree of difficulty. In Problem \ref{eq.1} and its dual given in Problem \ref{eq.2}, the degree of difficulty is $D=N-n-1$. Depending on $D$, we have the following two cases.
\begin{description}
	\item[I.] If $D \geq 0$, then Problem \ref{eq.2} is feasible. Because in this case, the number of equations is less than or equal to the number of dual variables. This guarantees the existence of a solution to Problem \ref{eq.2}. 
		\item[II.] On the other hand, if $D<0$, then Problem \ref{eq.2} is inconsistent. Because in this case, the number of equations is greater than the number of dual variables. It guarantees that there is no analytical solution to the dual problem. However, Sinha and Biswal \cite{sinha 1987} found an approximate solution to the dual problem with a negative degree of difficulty using the least squares method and the linear programming method.
\end{description} 
\end{remark}
\par In this section, our main aim is to develop an equivalent deterministic form of a GP problem in which the coefficients of the objective and constraint functions are linear-normal uncertain random variables. First, we consider the uncertain random GP problem as 
\begin{equation}\label{eq.4}
	\begin{aligned}
		&	\min \quad \tilde{f}_0{(\textbf{x})}=\sum\limits_{i=1}^{N_0}{\tilde{\beta}_{i0}}\prod\limits_{j=1}^{n}x_j^{\alpha_{0ij}}\\
		&	\text{s.t.}\\
		& \qquad \tilde{f}_k{(\textbf{x})}=\sum\limits_{i=1}^{N_k}{\tilde{\beta}_{ik}}\prod\limits_{j=1}^{n}x_j^{\alpha_{kij}}\leq 1,k=1,2,\ldots K,
	\end{aligned}
\end{equation}
where the coefficients ${\tilde{\beta}_{ik}}\sim \mathcal{LN}(A_{ik},B_{ik})$ are independent linear-normal uncertain random variables, $A_{ik}\sim\mathcal{N}(\mu_{A_{ik}},\sigma_{A_{ik}})$ and $B_{ik}\sim\mathcal{N}(\mu_{B_{ik}},\sigma_{B_{ik}})$ are independent normal random variables, $x_j>0$, $\alpha_{kij} \in \mathbb{R},\forall i,j,k$. To develop the equivalent deterministic form of Problem \ref{eq.4}, we first transform uncertain random variable coefficients ${\tilde{\beta}_{ik}}\sim \mathcal{LN}(A_{ik},B_{ik})$ into random variable coefficients based on Theorem \ref{thm.5}. For instance, Problem \ref{eq.4} is converted into stochastic GP problems via optimistic, pessimistic, and expected value criteria. Here are the transformations of Problem \ref{eq.4} into stochastic GP problems.
\begin{description}
	\item[I.] The transformation of Problem \ref{eq.4} into the stochastic GP problem via $\alpha$ optimistic value criteria is 
	\begin{equation}\label{eq.5}
		\begin{aligned}
			&	\min \quad \sum\limits_{i=1}^{N_0}{\tilde{\beta}}_{{i0}_{\sup}}(\alpha)\prod\limits_{j=1}^{n}x_j^{\alpha_{0ij}}\\
			&	\text{s.t.}\\
			& \qquad \sum\limits_{i=1}^{N_k}{\tilde{\beta}}_{{ik}_{\sup}}(\alpha)\prod\limits_{j=1}^{n}x_j^{\alpha_{kij}}\leq 1,k=1,2,\ldots K,
		\end{aligned}
	\end{equation}
where ${\tilde{\beta}}_{{ik}_{\sup}}(\alpha)\sim \mathcal{N}\big(\alpha\mu_{A_{ik}}+(1-\alpha)\mu_{B_{ik}},\sqrt{\alpha^2\sigma_{A_{ik}}^2+(1-\alpha)^2\sigma_{B_{ik}}^2}\big)$ are independent normal random variables, $\alpha\in (0,1)$.
\item[II.] The transformation of Problem \ref{eq.4} into the stochastic GP problem via $\alpha$ pessimistic value criteria is 
\begin{equation}\label{eq.6}
	\begin{aligned}
		&	\min \quad \sum\limits_{i=1}^{N_0}{\tilde{\beta}}_{{i0}_{\inf}}(\alpha)\prod\limits_{j=1}^{n}x_j^{\alpha_{0ij}}\\
		&	\text{s.t.}\\
		& \qquad \sum\limits_{i=1}^{N_k}{\tilde{\beta}}_{{ik}_{\inf}}(\alpha)\prod\limits_{j=1}^{n}x_j^{\alpha_{kij}}\leq 1,k=1,2,\ldots K,
	\end{aligned}
\end{equation}
where ${\tilde{\beta}}_{{ik}_{\inf}}(\alpha)\sim \mathcal{N}\big((1-\alpha)\mu_{A_{ik}}+\alpha\mu_{B_{ik}},\sqrt{(1-\alpha)^2\sigma_{A_{ik}}^2+\alpha^2\sigma_{B_{ik}}^2}\big)$ are independent normal random variables, $\alpha\in (0,1)$.
\item[III.] The transformation of Problem \ref{eq.4} into the stochastic GP problem via expected value criteria is 
\begin{equation}\label{eq.7}
	\begin{aligned}
		&	\min \quad \sum\limits_{i=1}^{N_0}{\tilde{\beta}}_{{i0}_{exp}}\prod\limits_{j=1}^{n}x_j^{\alpha_{0ij}}\\
		&	\text{s.t.}\\
		& \qquad \sum\limits_{i=1}^{N_k}{\tilde{\beta}}_{{ik}_{exp}}\prod\limits_{j=1}^{n}x_j^{\alpha_{kij}}\leq 1,k=1,2,\ldots K,
	\end{aligned}
\end{equation}
where ${\tilde{\beta}}_{{ik}_{exp}}\sim \mathcal{N}\bigg(\frac{\mu_{A_{ik}}+\mu_{B_{ik}}}{2},\frac{\sqrt{\sigma_{A_{ik}}^2+\sigma_{B_{ik}}^2}}{2}\bigg)$ are independent normal random variables.
\end{description}
We use probabilistic constraints with a tolerance level $\epsilon \in (0,0.5].$ Then Problem \ref{eq.5} becomes 
\begin{equation}\label{eq.8}
	\begin{aligned}
		&	\min x_0 \\
		&	\text{s.t.}\\
		& \qquad Pr\bigg(\sum\limits_{i=1}^{N_k}{\tilde{\beta}}_{{ik}_{\sup}}(\alpha)\prod\limits_{j=1}^{n}x_j^{\alpha_{kij}}\leq 1\bigg)\ge 1-\epsilon,k=1,2,\ldots K,\\
		& \qquad Pr\bigg(\sum\limits_{i=1}^{N_0}{\tilde{\beta}}_{{i0}_{\sup}}(\alpha)\prod\limits_{j=1}^{n}x_j^{\alpha_{0ij}}x_0^{-1}\le1\bigg)\ge 1-\epsilon,
	\end{aligned}
\end{equation}
Problem \ref{eq.6} becomes
\begin{equation}\label{eq.9}
	\begin{aligned}
		&	\min \quad x_0\\ 
		&	\text{s.t.}\\
		& \qquad Pr\bigg(\sum\limits_{i=1}^{N_k}{\tilde{\beta}}_{{ik}_{\inf}}(\alpha)\prod\limits_{j=1}^{n}x_j^{\alpha_{kij}}\leq 1\bigg)\ge1-\epsilon,k=1,2,\ldots K,\\
		& \qquad Pr\bigg(\sum\limits_{i=1}^{N_0}{\tilde{\beta}}_{{i0}_{\inf}}(\alpha)\prod\limits_{j=1}^{n}x_j^{\alpha_{0ij}}x_0^{-1}\le1\bigg)\ge 1-\epsilon,
	\end{aligned}
\end{equation}
Problem \ref{eq.7} becomes
\begin{equation}\label{eq.10}
	\begin{aligned}
		&	\min \quad x_0 \\
		&	\text{s.t.}\\
		& \qquad Pr\bigg(\sum\limits_{i=1}^{N_k}{\tilde{\beta}}_{{ik}_{exp}}\prod\limits_{j=1}^{n}x_j^{\alpha_{kij}}\leq 1\bigg)\ge 1-\epsilon,k=1,2,\ldots K,\\
		& \qquad Pr\bigg(\sum\limits_{i=1}^{N_0}{\tilde{\beta}}_{{i0}_{exp}}\prod\limits_{j=1}^{n}x_j^{\alpha_{0ij}}x_0^{-1}\le 1\bigg)\ge 1-\epsilon.
	\end{aligned}
\end{equation}
Now, the main focus is to find an equivalent deterministic form of a probabilistic constraint. The following theorem is very useful in this regard.
\begin{theorem}\label{thm.6}
		If $\xi_l\sim\mathcal{N}(\mu_l,\sigma_l)$ with $l=1,2,\ldots,L$ are independent normal random variables, and $U_l$ with $l=1,2,\ldots,L$ are nonnegative variables, then for every tolerance level $\epsilon \in(0,0.5]$, 
		\begin{align*}
			Pr\bigg(\sum\limits_{l=1}^{L}\xi_lU_l\le 1\bigg)\ge 1-\epsilon
			\end{align*}
		is equivalent to
		\begin{align*}
			\sum\limits_{l=1}^{L}\mu_lU_l+\Phi^{-1}(1-\epsilon)\sqrt{\sum\limits_{l=1}^{L}\sigma_l^2U_l^2}\le 1,
		\end{align*}
	where $\Phi^{-1}(1-\epsilon)=\sqrt{2}\text{erf}^{-1}(1-2\epsilon)$ is the inverse of the standard normal distribution function.
	\end{theorem}
\begin{proof}
	Let $Z=\bigg(\sum\limits_{l=1}^{L}\xi_lU_l\bigg)-1.$ Then we have 
	$\mu_Z=E[Z]=E\bigg[\bigg(\sum\limits_{l=1}^{L}\xi_lU_l\bigg)-1\bigg]=\bigg(\sum\limits_{l=1}^{L}E[\xi_l]U_l\bigg)-1=\sum\limits_{l=1}^{L}\mu_lU_l-1$ and $\sigma_Z^2=var(Z)=var\bigg[\bigg(\sum\limits_{l=1}^{L}\xi_lU_l\bigg)-1\bigg]=\bigg(\sum\limits_{l=1}^{L}var[\xi_l]U_l^2\bigg)-0=\sum\limits_{l=1}^{L}\sigma_l^2U_l^2.$ Therefore, we have 
	\begin{equation*}
		\begin{aligned}
		Pr\bigg(\sum\limits_{l=1}^{L}\xi_lU_l\le 1\bigg)\ge 1-\epsilon
		 & \Leftrightarrow Pr(Z\le 0)\ge 1-\epsilon\\
	    & \Leftrightarrow Pr\bigg(\frac{Z-\mu_Z}{\sigma_Z}\le -\frac{\mu_Z}{\sigma_Z}\bigg)\ge 1-\epsilon\\
	    & \Leftrightarrow \Phi\bigg(-\frac{\mu_Z}{\sigma_Z}\bigg)\ge 1-\epsilon \qquad \bigg[\because \frac{Z-\mu_Z}{\sigma_Z}\sim \mathcal{N}(0,1)\bigg]\\ 
	    & \Leftrightarrow -\frac{\mu_Z}{\sigma_Z}\ge \Phi^{-1}(1-\epsilon)\\
	    & \Leftrightarrow \mu_Z+\sigma_Z \Phi^{-1}(1-\epsilon)\le0\\
	    & \Leftrightarrow \sum\limits_{l=1}^{L}\mu_lU_l-1+\Phi^{-1}(1-\epsilon)\sqrt{\sum\limits_{l=1}^{L}\sigma_l^2U_l^2}\le0\\
	    & \Leftrightarrow \sum\limits_{l=1}^{L}\mu_lU_l+\Phi^{-1}(1-\epsilon)\sqrt{\sum\limits_{l=1}^{L}\sigma_l^2U_l^2}\le1.
		\end{aligned}
	\end{equation*}
This completes the proof.
\end{proof}
\par Based on Theorem \ref{thm.6}, we convert each probabilistic constraint into a deterministic constraint. Therefore, converting probabilistic constraints into deterministic constraints and, after simplification, the deterministic form of Problem \ref{eq.8}, Problem \ref{eq.9}, and Problem \ref{eq.10} are derived in Problem \ref{eq.11}, Problem \ref{eq.12}, and Problem \ref{eq.13}, respectively. Problem \ref{eq.11} is
\begin{equation}\label{eq.11}
	\begin{aligned}
		&	\min \sum\limits_{i=1}^{N_0}\big(\alpha\mu_{A_{i0}}+(1-\alpha)\mu_{B_{i0}}\big)\prod\limits_{j=1}^{n}x_j^{\alpha_{0ij}}+\Phi^{-1}(1-\epsilon)\sqrt{\sum\limits_{i=1}^{N_0}\big(\alpha^2\sigma_{A_{i0}}^2+(1-\alpha)^2\sigma_{B_{i0}}^2\big)\prod\limits_{j=1}^{n}x_j^{2\alpha_{0ij}}} \\
		&	\text{s.t.}\\
		& \quad \sum\limits_{i=1}^{N_k}\big(\alpha\mu_{A_{ik}}+(1-\alpha)\mu_{B_{ik}}\big)\prod\limits_{j=1}^{n}x_j^{\alpha_{kij}}+\Phi^{-1}(1-\epsilon)\sqrt{\sum\limits_{i=1}^{N_k}\big(\alpha^2\sigma_{A_{ik}}^2+(1-\alpha)^2\sigma_{B_{ik}}^2\big)\prod\limits_{j=1}^{n}x_j^{2\alpha_{kij}}}\leq 1,\\
		& \quad \text{where } \alpha \in (0,1), \epsilon \in(0,0.5], k=1,2,\ldots K.\\
	\end{aligned}
\end{equation}
Problem \ref{eq.12} is
\begin{equation}\label{eq.12}
	\begin{aligned}
		&	\min \sum\limits_{i=1}^{N_0}\big((1-\alpha)\mu_{A_{i0}}+\alpha\mu_{B_{i0}}\big)\prod\limits_{j=1}^{n}x_j^{\alpha_{0ij}}+\Phi^{-1}(1-\epsilon)\sqrt{\sum\limits_{i=1}^{N_0}\big((1-\alpha)^2\sigma_{A_{i0}}^2+\alpha^2\sigma_{B_{i0}}^2\big)\prod\limits_{j=1}^{n}x_j^{2\alpha_{0ij}}} \\
		&	\text{s.t.}\\
		& \quad \sum\limits_{i=1}^{N_k}\big((1-\alpha)\mu_{A_{ik}}+\alpha\mu_{B_{ik}}\big)\prod\limits_{j=1}^{n}x_j^{\alpha_{kij}}+\Phi^{-1}(1-\epsilon)\sqrt{\sum\limits_{i=1}^{N_k}\big((1-\alpha)^2\sigma_{A_{ik}}^2+\alpha^2\sigma_{B_{ik}}^2\big)\prod\limits_{j=1}^{n}x_j^{2\alpha_{kij}}}\leq 1,\\
		& \quad \text{where } \alpha \in (0,1), \epsilon \in(0,0.5], k=1,2,\ldots K.\\
	\end{aligned}
\end{equation}
Problem \ref{eq.13} is
\begin{equation}\label{eq.13}
	\begin{aligned}
		&	\min \sum\limits_{i=1}^{N_0}\bigg(\frac{\mu_{A_{i0}}+\mu_{B_{i0}}}{2}\bigg)\prod\limits_{j=1}^{n}x_j^{\alpha_{0ij}}+\frac{\Phi^{-1}(1-\epsilon)}{2}\sqrt{\sum\limits_{i=1}^{N_0}\big(\sigma_{A_{i0}}^2+\sigma_{B_{i0}}^2\big)\prod\limits_{j=1}^{n}x_j^{2\alpha_{0ij}}} \\
		&	\text{s.t.}\\
		& \quad \sum\limits_{i=1}^{N_k}\bigg(\frac{\mu_{A_{ik}}+\mu_{B_{ik}}}{2}\bigg)\prod\limits_{j=1}^{n}x_j^{\alpha_{kij}}+\frac{\Phi^{-1}(1-\epsilon)}{2}\sqrt{\sum\limits_{i=1}^{N_k}\big(\sigma_{A_{ik}}^2+\sigma_{B_{ik}}^2\big)\prod\limits_{j=1}^{n}x_j^{2\alpha_{kij}}}\leq 1,\\
		& \quad \text{where } \epsilon \in(0,0.5], k=1,2,\ldots K.\\
	\end{aligned}
\end{equation}
\section{Numerical example}\label{sec.5} 
In this section, we solve a numerical example to show the efficiency and efficacy of the procedures. We consider a GP problem with linear-normal uncertain random variable coefficients as 
\begin{equation}\label{eq.14}
	\begin{aligned}
		&	\min \quad \tilde{f}_0(\textbf{x})=\tilde{\beta}_{10}x_1^{-1} x_2^{-1}x_3^{-1}+\tilde{\beta}_{20}x_1 x_3\\
		&	\text{s.t.} \\
		&\qquad \tilde{f}_1(\textbf{x})= \tilde{\beta}_{11}x_1 x_2+\tilde{\beta}_{21}x_2 x_3\leq 1,\\
		&\qquad	x_1,x_2,x_3>0,
	\end{aligned}
\end{equation}
where $\tilde{\beta}_{10}\sim\mathcal{LN}(A_{10},B_{10})$, $\tilde{\beta}_{20}\sim\mathcal{LN}(A_{20},B_{20})$, $\tilde{\beta}_{11}\sim\mathcal{LN}(A_{11},B_{11})$, $\tilde{\beta}_{21}\sim\mathcal{LN}(A_{21},B_{21})$ are independent linear-normal uncertain random variables, $A_{10}\sim\mathcal{N}(50,3)$, $B_{10}\sim\mathcal{N}(40,2)$, $A_{20}\sim\mathcal{N}(45,2)$, $B_{20}\sim\mathcal{N}(40,1)$, $A_{11}\sim\mathcal{N}(1,\frac{1}{3})$, $B_{11}\sim\mathcal{N}(\frac{3}{2},\frac{2}{3})$, $A_{21}\sim\mathcal{N}(\frac{2}{3},\frac{1}{3})$, $B_{21}\sim\mathcal{N}(\frac{4}{3},1)$ are independent normal random variables. Based on an $\alpha$ optimistic value point of view, the deterministic form of Problem \ref{eq.14} with a tolerance level $\epsilon$ is
\begin{equation}\label{eq.15}
	\begin{aligned}
		&	\min\quad (40+10\alpha) x_1^{-1} x_2^{-1}x_3^{-1}+(40+5\alpha)x_1 x_3\\
		& \qquad+\Phi^{-1}(1-\epsilon)\sqrt{\big(9\alpha^2+4(1-\alpha)^2\big) x_1^{-2} x_2^{-2}x_3^{-2}+\big(4\alpha^2+(1-\alpha)^2\big)x_1^2 x_3^2}\\
		&	\text{s.t.} \\
		&\quad  \frac{3-\alpha}{2}x_1 x_2+\frac{4-2\alpha}{3}x_2 x_3\\
		&\qquad+\Phi^{-1}(1-\epsilon)\sqrt{\bigg(\frac{1}{9}\alpha^2+\frac{4}{9}(1-\alpha)^2\bigg)x_1^2 x_2^2+\bigg(\frac{1}{9}\alpha^2+(1-\alpha)^2\bigg)x_2^2 x_3^2}\leq 1,\\
		&\quad	x_1,x_2,x_3>0, \alpha\in(0,1), \epsilon \in (0,0.5].
	\end{aligned}
\end{equation} 
To solve Problem \ref{eq.15}, let us introduce two new variables, $x_4$ and $x_5$, defined as
 $x_4=\big(9\alpha^2+4(1-\alpha)^2\big) x_1^{-2} x_2^{-2}x_3^{-2}+\big(4\alpha^2+(1-\alpha)^2\big)x_1^2 x_3^2$ and $x_5=\big(\frac{1}{9}\alpha^2+\frac{4}{9}(1-\alpha)^2\big)x_1^2 x_2^2+\big(\frac{1}{9}\alpha^2+(1-\alpha)^2\big)x_2^2 x_3^2$. Then, Problem \ref{eq.15} becomes as 
\begin{equation}\label{eq.16}
	\begin{aligned}
		&	\min\quad (40+10\alpha) x_1^{-1} x_2^{-1}x_3^{-1}+(40+5\alpha)x_1 x_3+\Phi^{-1}(1-\epsilon)x_4^{\frac{1}{2}}\\
		&	\text{s.t.} \\
		&\quad  \frac{3-\alpha}{2}x_1 x_2+\frac{4-2\alpha}{3}x_2 x_3+\Phi^{-1}(1-\epsilon)x_5^{\frac{1}{2}}\leq 1,\\
		& \quad \big(9\alpha^2+4(1-\alpha)^2\big) x_1^{-2} x_2^{-2}x_3^{-2}x_4^{-1}+\big(4\alpha^2+(1-\alpha)^2\big)x_1^2 x_3^2x_4^{-1}\le1,\\
		& \quad \bigg(\frac{1}{9}\alpha^2+\frac{4}{9}(1-\alpha)^2\bigg)x_1^2 x_2^2x_5^{-1}+\bigg(\frac{1}{9}\alpha^2+(1-\alpha)^2\bigg)x_2^2 x_3^2x_5^{-1}\le 1,\\
		&\quad	x_1,x_2,x_3,x_4,x_5>0, \alpha\in(0,1), \epsilon \in (0,0.5].
	\end{aligned}
\end{equation}
In Problem \ref{eq.16}, the total number of terms is $10$ and the total number of variables is $5.$ So, the degree of difficulty is $=10-5-1=4.$ Let $\delta_1,\delta_2,\ldots,\delta_{10}$ be the corresponding dual variables. The the dual of Problem \ref{eq.16} is
\begin{equation}\label{eq.17}
	\begin{aligned}
		&\max \quad V(\delta)=\bigg(\frac{40+10\alpha}{{\delta_{1}}}\bigg)^{{\delta_{1}}}\bigg(\frac{40+5\alpha}{{\delta_{2}}}\bigg)^{{\delta_{2}}} \bigg(\frac{\Phi^{-1}(1-\epsilon)}{{\delta_{3}}}\bigg)^{{\delta_{3}}}\bigg(\frac{3-\alpha}{{2\delta_{4}}}\bigg)^{{\delta_{4}}}\bigg(\frac{4-2\alpha}{{3\delta_{5}}}\bigg)^{{\delta_{5}}} \\ 
		& \quad
		\bigg(\frac{\Phi^{-1}(1-\epsilon)}{{\delta_{6}}}\bigg)^{{\delta_{6}}}
		\bigg(\frac{9\alpha^2+4(1-\alpha)^2}{{\delta_{7}}}\bigg)^{{\delta_{7}}} \bigg(\frac{4\alpha^2+(1-\alpha)^2}{{\delta_{8}}}\bigg)^{{\delta_{8}}}
		\bigg(\frac{\alpha^2+4(1-\alpha)^2}{{9\delta_{9}}}\bigg)^{{\delta_{9}}}\\
		& \quad
		\bigg(\frac{\alpha^2+9(1-\alpha)^2}{{9\delta_{10}}}\bigg)^{{\delta_{10}}}(\delta_{4}+\delta_{5}+\delta_{6})^{\delta_{4}+\delta_{5}+\delta_{6}} (\delta_{7}+\delta_{8})^{\delta_{7}+\delta_{8}} (\delta_{9}+\delta_{10})^{\delta_{9}+\delta_{10}}\\
		&\text{s.t.}\\
		& \qquad \delta_1+\delta_{2}+\delta_{3}=1, \\
		&\qquad -\delta_{1}+\delta_{2}+\delta_{4}-2\delta_{7}+2\delta_{8}+2\delta_{9}=0,\\
		&\qquad -\delta_{1}+\delta_{4}+\delta_{5}-2\delta_{7}+2\delta_{9}+2\delta_{10}=0,\\
		&\qquad -\delta_{1}+\delta_{2}+\delta_{5}-2\delta_{7}+2\delta_{8}+2\delta_{10}=0,\\
		&\qquad \frac{1}{2}\delta_{3}-\delta_{7}-\delta_{8}=0, \\
		&\qquad \frac{1}{2}\delta_{6}-\delta_{9}-\delta_{10}=0.
	\end{aligned}
\end{equation}
Solving Problem \ref{eq.17} for a fixed $\alpha \in (0,1)$ with the tolerance level $\epsilon=0.05$, we get the dual solution. Consequently, by the primal-dual relationship, we find the primal solution. When we set parameter $\alpha$ to discrete values from $0.1$ to $0.9$ with an increment $0.1$, the corresponding solutions are computed in Table \ref{table.1}. \\
\begin{table}[htbp]
	\caption{ Optimal solutions }
	\label{table.1}
	\centering
	\begin{tabular}{c c c c c c c }
		\hline\hline
		$\alpha$ & $x_1^*$ & $x_2^*$ & $x_3^*$ & $x_4^*$ & $x_5^*$  & Objective value\\ [0.5ex]
		\hline\hline
		0.1 & 1.454 & 0.167 & 1.233 & 39.882 & 0.056 & 219.893 \\
		0.2 & 1.409 & 0.183 & 1.219 & 31.919 & 0.051 & 213.316\\
		0.3 & 1.361 & 0.201 & 1.207 & 27.584 & 0.047 & 206.732 \\
		0.4 & 1.310 & 0.223 & 1.199 & 26.036 & 0.042 & 200.180\\
		0.5 & 1.254 & 0.247 & 1.195 & 26.533 & 0.038 & 193.715\\
		0.6 & 1.194 & 0.274 & 1.197 & 28.480 & 0.034 & 187.493 \\
		0.7 & 1.132 & 0.304 & 1.208 & 31.494 & 0.031 & 181.885\\
		0.8 & 1.075 & 0.332 & 1.225 & 35.514 & 0.030 & 177.570 \\
		0.9 & 1.033 & 0.354 & 1.242 & 40.852 & 0.032 & 175.417 \\ [1ex]
		\hline

	\end{tabular}
	
\end{table}\\
We find the deterministic form of Problem \ref{eq.14} based on an optimistic value point of view. Further, we solve the transformed deterministic problem for different values of $\alpha\in(0,1)$, which are given in Table \ref{table.1}. Similarly, it is easy to find the solution to Problem \ref{eq.14} based on the pessimistic value point of view and the expected value point of view. However, from Table \ref{table.1}, we visualize the solution to Problem \ref{eq.14} based on the pessimistic value point of view and the expected value point of view. Figure \ref{fig:6} shows the objective values with respect to the parameter $\alpha$ based on the optimistic value point of view and the pessimistic value point of view. However, $\alpha=0.5$ gives the optimal objective value based on expected value point of view.
\begin{figure}[h] 
	\centering 
	\includegraphics[scale=0.2]{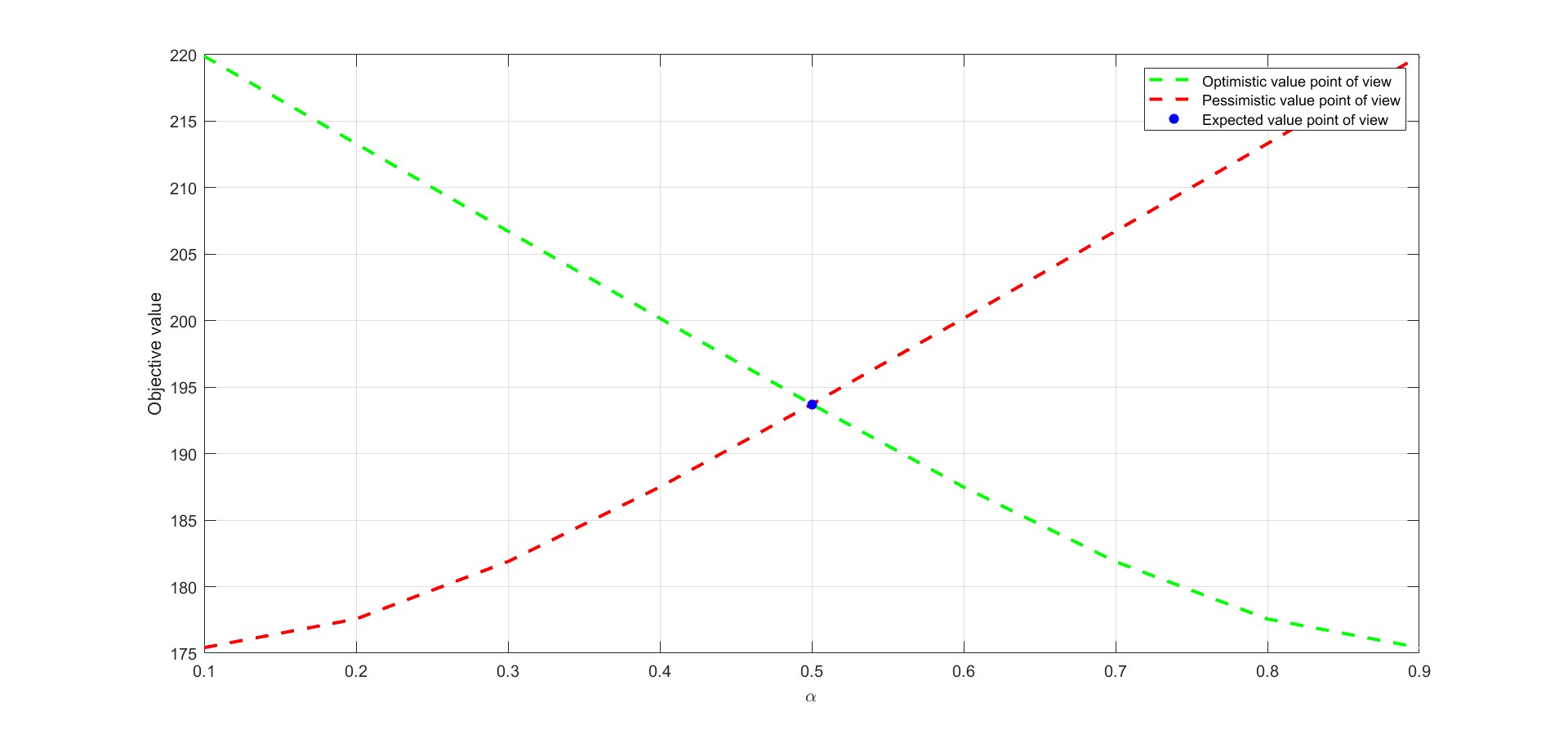} 
	\caption{Optimal objective values with respect to the parameter $\alpha$.}
	\label{fig:6}
\end{figure}
\section{Conclusion}\label{sec.6}
 In conclusion, our work delves into the realm of GP from a unique perspective by incorporating independent linear-normal uncertain random variables into the problem. We introduce and define the concept of an uncertain random variable and propose a novel approach in the form of linear-normal uncertain random variables. We are able to transform these uncertain random variables into their stochastic counterparts by using three different criteria: optimistic value, pessimistic value, and expected value. This turns the uncertain random GP problem into a manageable stochastic GP problem.
 
\par One of the primary advantages of our approach is its ability to model and address uncertainty in a structured manner, which is particularly useful in real-world scenarios where data is not always deterministic. This method allows decision-makers to make informed choices that consider the inherent uncertainty and randomness in the input coefficients, ultimately leading to more reliable solutions.
 
 \par Furthermore, we present an equivalent deterministic representation of the transformed GP problem, which simplifies the decision-making process for practitioners who may prefer working with deterministic models. This duality between stochastic and deterministic representations provides flexibility in addressing the problem based on the decision-maker's preference and data availability.
 
 \par However, it is essential to acknowledge the limitations of our work. Firstly, the computational complexity of solving stochastic GP problems can be substantial, especially for large-scale instances, and may require specialized optimization techniques or approximation methods. Secondly, our approach relies on the assumption that the coefficients are independent linear-normal uncertain random variables, which might not always hold in practical applications. Addressing correlated uncertainties remains an avenue for future research. Additionally, the transformation of uncertain variables into stochastic forms may introduce approximations, potentially impacting the accuracy of the results.
 
 \par In summary, our work offers a valuable contribution to the field of GP problems by providing a systematic approach to handling uncertainty and randomness through linear-normal uncertain random variables. While it offers practical benefits in decision-making under uncertainty, decision-makers should be aware of the computational challenges and the assumptions inherent in our approach. Future research could explore more sophisticated techniques to handle correlated uncertainties and further improve the applicability of stochastic GP problems in real-world scenarios.\\\\
{\small \textbf{Acknowledgement }The first author is thankful to CSIR for financial support of this work through file No: 09\textbackslash 1059(0027)\textbackslash 2019-EMR-I.}
\\
\textbf{Data availability } Data sharing not applicable to this article as no datasets were generated. A random data set is taken for the numerical example given in this paper.\\
\textbf{Conflict of interests} The authors have no relevant financial or non financial interests to disclose. The authors
declare that they have no conflict of interests.

\end{document}